\documentclass{amsart}
\usepackage{amsfonts,amssymb,amscd,amsmath,enumerate,verbatim,calc}
\usepackage[all]{xy}

\newcommand{\CM}{Cohen-Macaulay}

\newcommand{\wrt}{with respect to}

\newcommand{\n}{\mathfrak{n} }
\newcommand{\m}{\mathfrak{m} }
\newcommand{\M}{\mathfrak{M} }

\newcommand{\R}{\mathcal{R} }
\newcommand{\Z}{\mathbb{Z} }
\newcommand{\C}{\mathcal{C} }
\newcommand{\rt}{\rightarrow}

\newcommand{\ov}{\overline}

\newcommand{\image}{\operatorname{image}}

\newcommand{\Syz}{\operatorname{Syz}}

\newcommand{\Ass}{\operatorname{Ass}}

\newcommand{\Hom}{\operatorname{Hom}}
\newcommand{\sHom}{\operatorname{\underline{Hom}}}

\newcommand{\Ext}{\operatorname{Ext}}

\theoremstyle{plain}

\newtheorem{theorem}{Theorem}[section]
\newtheorem{corollary}[theorem]{Corollary}
\newtheorem{lemma}[theorem]{Lemma}

\theoremstyle{definition}

\newtheorem{s}[theorem]{}
\newtheorem{remark}[theorem]{Remark}

\theoremstyle{remark}

\begin{document}

\title[Bockstein cohomology]{Bockstein cohomology of Maximal Cohen-Macaulay modules over Gorenstein isolated singularities}
\author{Tony~J.~Puthenpurakal}
\date{\today}
\address{Department of Mathematics, IIT Bombay, Powai, Mumbai 400 076}

\email{tputhen@math.iitb.ac.in}

 \begin{abstract}
Let $(A,\mathfrak{m})$ be an excellent equi-charateristic Gorenstein isolated singularity of dimension $d \geq 2$. Assume the residue field of $A$ is perfect.
Let $I$ be any $\m$-primary ideal. Let $G_I(A) = \bigoplus_{n \geq 0}I^n/I^{n+1}$ be the associated graded ring of $A$ with respect to $I$ and let $\mathcal{R}_I(A) = \bigoplus_{n \in \mathbb{Z}}I^n$ be the extended Rees algebra of $A$ with respect to $I$. Let $M$ be a finitely generated $A$-module.
 Let $G_I(M) = \bigoplus_{n \geq 0}I^nM/I^{n+1}M$ be the associated graded ring of $M$ with respect to $I$ (considered as a $G_I(A)$-module). Let $BH^i(G_I(M))$ be the $i^{th}$-Bockstein cohomology  of $G_I(M)$ with respect to $\mathcal{R}_I(A)_+$-torsion functor.
 We show there exists $a \geq 1$ depending only on $A$
such that if $I$ is any $\m$-primary ideal with
$I \subseteq \m^a$ and  $G_I(A) $ generalized Cohen-Macaulay then the Bockstein cohomology $BH^i(G_I(M))$ has finite length for $i = 0, \ldots, d-1$ for any maximal Cohen-Macaulay $A$-module $M$.
\end{abstract}
 \maketitle
\section{introduction}
Let $(A,\m)$ be a \CM \ local ring of dimension $d$ and let $I$ be an $\m$-primary ideal. Let $G_I(A) = \bigoplus_{n \geq 0}I^n/I^{n+1}$ be the associated graded ring of $A$ with respect to $I$ and let $\R_I(A) = \bigoplus_{n \in \Z}I^n$ be the extended Rees algebra of $A$ with respect to $I$. Let $M$ be a \CM \ $A$-module.
 Let $G_I(M) = \bigoplus_{n \geq 0}I^nM/I^{n+1}M$ be the associated graded ring of $M$ with respect to $I$ (considered as a $G_I(A)$-module) and let $\R_I(M) = \bigoplus_{n \in \Z}I^nM$ be the extended Rees module of $M$ with respect to $I$.

 The \emph{Hilbert function} of $M$ \wrt \ $I$ is
$H^I(M,n) = \ell(I^nM/I^{n+1}M)$.  Here $\ell(-)$ denotes length as an $A$-module. A fruitful area of research has been to study the interplay between Hilbert functions and properties of
$G_I(M)$ and $\R_I(M)$.  See the texts
 \cite[Section 6]{VaSix} and  \cite[Chapter 5]{VasBook} for nice surveys on this subject (when $M = A$).
 Traditionally only the case $M = A$ was considered. However recently associated graded modules have been studied, see \cite{rv}.

Graded local cohomology  has played an important role in this subject.
 For  various applications
   see  \cite[4.4.3]{BH}, \cite{Durham}, \cite{VerJoh},
\cite{Blanc},  \cite{ItN}, \cite{Tr}
   and \cite{HMc}.
 Let $H^i(G_I(M))$ denote $i^{th}$-local cohomology module of $G_I(A)$ \wrt \ $G_I(A)_+ = \bigoplus_{n>0}I^n/I^{n+1}$. A line of inquiry in this subject is to find conditions on $I$ such that $G_I(A)$ (or $G_{I^n}(A)$ for all $n \gg 0$) has high depth. This is equivalent to showing that $H^i(G_I(A))$ (or $H^i(G_{I^n}(A))$ for all $n \gg 0$) vanishes for some $i < d$.

 We note that $t^{-1}$ is $\R_I(M)$-regular and $\R_I(M)/t^{-1}\R_I(M) = G_I(M)$. So we have naturally defined Bockstein operators
  $\beta^i \colon H^i(G_I(M))(-1) \rt H^{i+1}(G_I(M))$ for $i \geq 0$ (with respect to $\R_I(A)_+$-torsion functor). Since $\beta^{i+1}(+1)\circ \beta^i = 0$ we have \textit{Bockstein cohomology} modules $BH^i(G_I(M))$ for $i = 0,\ldots,d$. Despite being natural, Bockstein cohomology groups of associated graded rings have not been investigated much. In  \cite{PuB}  we studied some basic properties of Bockstein cohomology. We also showed that in some respects Bockstein cohomology behaves better than the usual local cohomology.

Maximal Cohen-Macaulay (MCM) modules encode a lot of information of the ring. In fact if $A$ is Gorenstein then the stable category of MCM $A$-modules is isomorphic to the singularity category of $A$, see \cite[4.4.1]{Buch}. A particularly important case is when $A$ is an isolated singularity (i.e., $A_P$ is regular for all primes $P \neq \m$). In this paper we take the view that information on $G_I(A)$ yields structural information on $G_I(M)$ when $M$ is MCM.

The main result of this paper is:
\begin{theorem}\label{main}
Let $(A,\mathfrak{m})$ be an excellent equi-characteristic \CM \ isolated singularity of dimension $d \geq 2$. Assume the residue field of $A$ is perfect. There exists $a \geq 1$ depending only on $A$
such that if $I$ is any $\m$-primary ideal with
$I \subseteq \m^a$ and  $H^i(G_I(A)) $ has finite length for $i < r$ then the Bockstein cohomology $BH^i(G_I(M))$ has finite length for $i < r$ for any MCM $A$-module $M$
such that $M = \Syz^A_1(L)$ where $L$ is a MCM $A$-module.
\end{theorem}

We note that if $A$ is Gorenstein then any MCM $A$-module is the syzygy of a MCM $A$-module. So we obtain as an easy corollary the result stated in the abstract.
\begin{corollary}\label{main-cor}
 Let $(A,\mathfrak{m})$ be an excellent equi-characteristic Gorenstein isolated singularity of dimension $d \geq 2$. Assume the residue field of $A$ is perfect. There exists $a \geq 1$ depending only on $A$
such that if $I$ is any $\m$-primary ideal with
$I \subseteq \m^a$ and  $G_I(A)$ generalized \CM \  then the Bockstein cohomology $BH^i(G_I(M))$ has finite length for $i < d$ for any MCM $A$-module $M$.
\end{corollary}

\s \emph{Techniques used to prove our result:}
It is elementary that we can reduce proof of Theorem \ref{main} to the case when $A$ is complete isolated singularity with infinite perfect residue field, see \ref{red}.  The main technique used is the notion of cohomological annhilators, i.e., there exists $a \geq 1$
such that $\m^a \Ext^1_A(X,Y) = 0$ for any MCM $A$-modules $X,Y$, see \cite[6.10]{Y} (also see \cite[15.14]{LW}). Let $\sHom_A(M, M)$ denote the stable Hom. If $M = \Syz^A_1(L)$ where $L$ is MCM $A$-module then $\m^a \sHom_A(M, M) = 0$. In particular if $I \subseteq \m^a$ then for any  $x \in I$,  the multiplication map $\mu_x \colon M \rt M$ factors through a free $A$-module. The assumptions on $G_I(A)$ also yield conditions on the local cohomology of $\R_I(A)$ with respect to its maximal homogeneous ideal. We then choose $x \in I$ sufficiently general to conclude.

We now describe in brief the contents of this paper. In section two we discuss some preliminaries on Bockstein cohomology that we need. In section three we discuss some preliminaries on excellent isolated Cohen-Macaulay local rings. We also show that we may reduce to the case when $A$ is complete with infinite perfect residue field. In section four we prove some results on the local cohomology of the extended Rees module with respect to the maximal homogeneous ideal of $\R_I(A)$. Finally in section five we prove Theorem \ref{main}.
\section{Bockstein Cohomology}
In this paper all rings are commutative Noetherian and all modules are assumed to be finitely generated unless specified otherwise.
In this section we first recall a very general construction of Bockstein cohomology. We then specialize to the case of associated graded modules.

\s \emph{General construction of  Bockstein Cohomology}.

 Let $R$ be a ring, $M$ an $R$-module and $x$ a regular element on $M$. We have a natural exact sequence
\[
0 \rightarrow \frac{M}{xM} \xrightarrow{\alpha} \frac{M}{x^2M} \xrightarrow{\pi} \frac{M}{xM} \rightarrow 0.
\]
Here $\pi$ is the natural projection map and $\alpha(m + xM) = xm + x^2M$. \\
Let $F \colon Mod(R) \rightarrow Mod(R)$ be any left exact functor. Then note that we have connecting homomorphisms
\[
\beta^i  \colon RF^i(M/xM) \rightarrow RF^{i+1}(M/xM).
\]
We call $\beta^i$ the $i^{th}$ \emph{Bockstein operator} on $M/xM$ with respect to $F$.

\s \label{short} Consider the natural exact sequence
\[
0 \rt M \xrightarrow{x} M \xrightarrow{\rho} M/xM \rt 0.
\]
So we have an exact sequence
\begin{equation*}
\rt RF^i(M/xM)\xrightarrow{\delta^i} RF^{i+1}(M) \rt RF^{i+1}(M) \xrightarrow{RF^{i+1}(\rho)} RF^{i+1}(M/xM) \rt \tag{$\dagger$}
\end{equation*}
It can be easily shown that $\beta^i =  RF^{i+1}(\rho)\circ \delta^i$.  Since $\delta^{i+1} \circ RF^{i+1}(\rho) = 0$ we get
that $\beta^{i+1} \circ \beta^{i} = 0$ for all $i \geq 0$. Thus we have a complex
\[
\cdots \xrightarrow{\beta^{i-1}} RF^i(M/xM) \xrightarrow{\beta^{i}} RF^{i+1}(M/xM) \xrightarrow{\beta^{i+1}} RF^{i+2}(M/xM) \cdots
\]
The cohomology of this complex is denoted by $BF^*(M/xM)$ and is called the \emph{Bockstein cohomology} of $M/xM$ with respect to $F$.

\s \textit{Bockstein Cohomology of Associated graded modules }

Let $\R_I(A) = \bigoplus_{n \in \mathbb{Z}}I^n$ be the extended Rees-ring of $A$ with respect to $I$.  Here $I^n = A$ for all $n \leq 0$ and $\R_I(A)$ is considered as a subring of $A[t,t^{-1}]$. Let $\R(I)_+$  be the ideal in $\R(I)$ generated by $ \bigoplus_{n >0}I^n$.
Let $M$ be an $A$-module. Let $\R_I(M) = \bigoplus_{n \in \mathbb{Z}}I^nM$ be the extended Rees-module of $M$ with respect to $I$.

Clearly $t^{-1}$ is a non-zero divisor on $\R_I(M)$. Note $\R_I(M)/t^{-1}\R_I(M) = G_I(M)$. We have an exact sequence (after a shift)
\[
0 \rightarrow G_I(M) \rightarrow \R_I(M)/t^{-2}\R_I(M)(-1) \rightarrow G_I(M)(-1) \rightarrow 0.
\]
Here
$$\frac{\R_I(M)}{t^{-2}\R_I(M)} = M/IM \oplus M/I^2M \oplus IM/I^3M\oplus I^2M/I^4M \oplus \cdots \oplus I^{n-1}M/I^{n+1}M \cdots,$$
with $M/IM$ sitting in degree $-1$.

\s Let $\Gamma_{\R_I(A)_+} \colon Mod(\R_I(A)) \rightarrow Mod(\R_I(A))$ be the $\R_I(A)_+$-torsion functor.  Set $G= G_I(A)$. So by the general theory we have Bockstein homomorphisms
\[
\beta^i \colon H^i_{G_+}(G_I(M))(-1) \rightarrow H^{i+1}_{G_+}(G_I(M)),
\]
and we have Bockstein cohomology modules
$$BH^i_{G_+}(G_I(M))  = \ker(\beta^{i}(+1))/\image(\beta^{i-1})   \quad \text{ for all $i \geq 0$}. $$
Set $\beta^i_I(M) = \beta^i(G_I(M))$.

\s \label{m-torsion} Assume $I$ is $\m$-primary. Let $\M$ be the maximal homogeneous ideal of $\R_I(A)$. We may consider the Bockstein cohomology modules with respect to $\M$-torsion functor. However we get the same maps since the natural map \\ $H^i_{\M}(E) \rt H^i_{\R_I(A)_+}(E)$ is an isomorphisms when
$E = G_I(A)$ or \\ $E = \R_I(M)/t^{-2}\R_I(M)(-1) $. However this is convenient as in \ref{short} (equation $(\dagger)$) we may work with $H^i_\M(\R_I(M))$ which is a $*$-Artinian $\R_I(A)$-module.
\begin{remark}\label{extn}
Let $(A,\m) \rt (A^\prime,\m^\prime)$ be a flat extension with $\m A^\prime = \m^\prime$. Set $I^\prime = IA^\prime$ and $M^\prime = M \otimes_A A^\prime$. Then it is clear that
$$\beta^i_{I^\prime}(M^\prime)  = \beta^i_I(M)\otimes_A A^\prime.$$
It follows that for all $i \geq 0$ we have
\[
BH^i_{G^\prime_+}(G_{I^\prime}(A^\prime)) \cong   BH^i_{G_+}(G_I(A))\otimes_A A^\prime.
\]
\end{remark}
\section{Some preliminaries on excellent isolated \CM \ local rings}
In our arguments we need to assume that $A$ is complete with infinite residue field (which is perfect).  We show how to reduce our Theorem to this case. We also give consequences of assuming $A$ is complete equi-characteristic \CM \  isolated singularity in terms of annhilators of $\Ext^1_A(X,Y)$ with $X, Y$ MCM $A$-modules.

\s \label{red}Let $A$ be an excellent equi-characteristic \CM \ isolated singularity. Assume the residue field of $A$ is perfect.
We wish to consider a  flat local extension $(A, \m) \rt (B, \n)$ with $\m B = \n$ such that $\m B = \n$, $B$ is complete, \CM \ isolated singularity with infinite perfect residue field.
Notice Bockstein cohomology behaves well under such extensions, see \ref{extn}.

Note $\widehat{A}$ is a  \CM \ isolated singularity containing a field isomorphic to $k = A/\m$.
If $k$ is infinite then we do  not have to do anything further and set $B = \widehat{A}$.

 If $k$ is finite the note that we \emph{cannot} do the standard technique of replacing $A$ with $B = A[X]_{\m A[X]}$ as the residue field of $B$ is $k(X)$ which is NOT perfect.
 So when $k$ is finite we do the following construction (from \cite[section 2]{PPG}): \\
 First complete $A$ (and so we assume $A$ is complete). Note $A$ will contain a field isomorphic to $k$. For convenience denote it by $k$ too. Fix an algebraic closure $\ov{k}$  of $k$. Let
 $$\C = \{ \ell \mid \ell \text{ is a finite extension of $k$ and contained in $\ov{k}$} \}.$$
 For $\ell \in \C$ let $A_\ell = A\otimes_k l$. Then $A_\ell$ is local with maximal ideal $\m_\ell = \m A_\ell$ and residue field $\ell$. Note clearly $\C$ is obviously a directed set (using inclusion). So we have a direct system of rings $\{ A_\ell \}_{\ell\in \C}$. Let $T$ be the direct limit. Then $T$ is a \CM \ local ring with maximal ideal $\m_T = \m T$ and residue field $\ov{k}$. If $A$ is Gorenstein then so is $T$, see \cite[3.4]{PPG}. Furthermore $T$ is excellent, see \cite[3.3]{PPG}. By \cite[10.7]{LW}, $T$ is also an isolated singularity. Set $B = \widehat{T}$.

 \s \label{ann} Assume $A$ is complete equicharacteristic \CM \ local ring which is an isolated singularity. Assume $k = A/\m$ is perfect. Then by \cite[6.10]{Y} (also see \cite[15.14]{LW})  it follows there exists $a \geq 1$
 such that $\m^a \Ext^1_A(X,Y) = 0$ for all maximal \CM \ $A$-modules.

 \s Let $M, N$ be $A$-modules. Let $\theta(M,N)$ be the $A$-submodule of  $\Hom_A(M, N)$ consisting of  all maps $f \colon M \rt N$ which factor through a free $A$-module. Set $\sHom_A(M, N) = \Hom_A(M,N)/\theta(M,N)$.

 With these preliminaries we have
 \begin{lemma}
\label{hom-ann}(with hypotheses as in \ref{ann}) Let $M, N$ be MCM $A$-modules. Assume $M = \Syz^A_1(L)$ where $L$ is a MCM $A$-module. Then $\m^a \sHom_A(M, N) = 0$.
 \end{lemma}
 \begin{proof}
   Consider an exact sequence $0 \rt M \rt F \rt L \rt 0$ where $F$ is a free $A$-module. So we have an exact sequence
   \[
   0 \rt \Hom_A(L, N) \rt  \Hom_A(F, N) \xrightarrow{\epsilon}  \Hom_A(M, N) \rt \Ext^1_A(L, N) \rt 0.
   \]
   It follows that $\Hom_A(M, N)/\image(\epsilon) \cong \Ext^1_A(L, N)$. The result follows as \\  $\sHom_A(M, N)$ is a quotient of $\Hom_A(M, N)/\image(\epsilon)$.
 \end{proof}
\section{Local cohomology of extended Rees module }
The aim of this section is to prove the following result which we need to prove Theorem \ref{main}.
\begin{theorem}\label{prelim}
Let $(A,\m)$ be complete, \CM \  local ring of dimension $d \geq 2$. Let $I$ be an $\m$-primary ideal in $A$. Let $\R_I(A)$ be the extended Rees-ring of $A$ \wrt\ $I$ and let $\M$ be its maximal homogeneous ideal. Let $N$ be a maximal \CM \ $A$-module. Then
\begin{enumerate}[\rm (1)]
  \item The natural map $H^i_\M(\R_I(N)) \rt H^i_{\R_I(A)_+}(\R_I(N))$ is an isomorphism for $i < d$.
  \item The natural map $H^d_\M(\R_I(N)) \rt H^d_{\R_I(A)_+}(\R_I(N))$ is an inclusion.
  \item For all $i \leq d$ we have $H^i_\M(\R_I(N))_n = 0$ for $n \gg 0$. Furthermore \\  $H^i_\M(\R_I(N))_n$ has finite length for all $n \in \Z$ and all $i \leq d$.
  \item If the residue field $k$ of $A$ is infinite then there exists $x\in I$ which is $N \oplus A$-superficial such that the map
      $H^i_\M(\R_I(N))(-1) \xrightarrow{xt} H^i_\M(\R_I(N))$ has finite length co-kernel (for all $i \leq d$).
  \item
   If $H^i_{G_I(A)_+}(G_I(N))$ has finite length for $i < r$ then $H^i_\M(\R_I(N))$ has finite length for $i < r +1$
\end{enumerate}
\end{theorem}
Before proving this theorem we need to discuss a few preliminaries that we need.
\s The following $\R_I(A)$-module (introduced in \cite{Pu5}) is very convenient to work with.
$$ L^I (N)= \bigoplus_{n \geq 0} \frac{N}{I^{n+1}N}.$$
To see that $L^I(M)$ is an $\R(I)$-module, note that we have an exact sequence
\[
0 \rt \R_I(N)\rt N[t,t^{-1}] \rt L^I(N)(-1) \rt 0.
\]
By this exact sequence we  give $L^I(N)$ a structure of $\R_I(A)$-module.
Note that $L^I(N)$ is \textit{not} finitely generated as a $\R_I(A)$-module.

\s \label{L} Let $N$ be a MCM $A$-module. Let $x_1, \ldots, x_d$ be an $A$-regular sequence. Then it a $N$-regular sequence. Notice $x_1t, \ldots, x_{d}t \in \R_I(A)_+$ is a $N[t,t^{-1}]$-regular sequence.Thus $H^i_{\R_I(A)_+}(N[t,t^{-1}]) = 0$ for $i < d$. It follows that $H^i_{\R_I(A)_+}(\R_I(N)) = H^{i-1}_{\R_I(A)_+}(L^I(N))(-1)$ for $i < d$.

\s \label{compare} In \cite[3.1]{PuB} we proved that the natural map $H^i_\M(L^I(N)) \rt H^i_{\R_I(A)_+}(L^I(N))$ is an isomorphism for all $i$.

The following lemma is useful.
\begin{lemma}\label{sup}
(with hypotheses as in \ref{prelim}) Assume the residue field of $k$ is infinite.
Let $E = \bigoplus_{n \in \Z} E_n$ be a finitely generated $\R_I(A)$-module with $E_n = 0$ for $n \ll 0$. Then $\ell_A(E_n)$ is finite for all $n \in \Z$. Furthemore there exists a non-empty Zariski open
 dense subset $U$ of $I/\m I$ such that if the image of $x \in I$ is in $U$ then the map $E(-1)\xrightarrow{xt} E$ has finite length kernel.
\end{lemma}
\begin{proof}
Let $E$ be generated by homogeneous elements $e_1, \ldots, e_s$. Then there exists $a$ with $t^{-a}e_i = 0$ for all $i$. So $t^{-a}E = 0$. Thus $E$ is a $S = \R_I(A)/t^{-a}\R_I(A)$-module. As $\ell_A(S_n)$ is finite for all $n \in \Z$ it follows that $\ell_A(E_n) $ is finite for all $n \in \Z$.

Let $L = It \cong I$ and set $V = L/\m L \cong I/\m I$. If $P $ is a prime ideal in $\R_I(A)$ then set  $P_1 = P \cap L $ and $\ov{P_1} = (P_1  + \m L)/\m L$. If $P \nsupseteq \R_I(A)_+$
then $P_1 \neq L$ and so $\ov{P_1}$ is a proper subspace of $V$. As $k$ is infinite
\[
U = V \setminus \bigcup_{\stackrel{P \in \Ass(E)}{P \nsupseteq \R_I(A)_+}}\ov{P_1}  \neq \emptyset.
\]
  Let $xt \in U$. Let $K = \ker(E(-1) \xrightarrow{xt} E)$. We prove $\ell(K) < \infty$.

  Claim-1: If $Q \in \Ass(K)$ then $Q \supseteq \R(I)_+$. \\
  Suppose if possible $Q \nsupseteq \R(I)_+$. Note $Q \in \Ass(E)$. So $xt \notin Q$.
  We have an exact sequence
  $$ 0 \rt K_Q \rt E_Q \xrightarrow{xt} E_Q.$$
  As $xt \notin Q$, it is a unit in $\R_I(A)_Q$. It follows that $K_Q = 0$, a contradiction.

  By claim 1 it follows that if $w \in K$ then $(\R_I(A)_+)^mw = 0$ for some $m > 0$. Thus $K$ is $\R_I(A)_+$-torsion. Ass $\ell(K_n) < \infty $ for all $n$ it follows that $K$ is $(\m t^0)$-torsion. Also as $t^{-a}K = 0$ (as $t^{-a}E = 0$) it follows that $K$ is $t^{-1}$-torsion. So $K$ is $\M$-torsion. As $K$ is finitely generated it follows that $K$ has finite length.
  The result follows.
\end{proof}

We mow give
\begin{proof}[Proof of Theorem \ref{prelim}]
  As $I$ is $\m$-primary it follows that $\sqrt{(t^{-1}, \R_I(A)_+)} = \M$. We have a long exact sequence in cohomology
  \[
  H^{i-1}_{\R_I(A)}(\R_I(N))_{t^{-1}} \rt  H^{i}_{\M}(\R_I(N)) \rt  H^{i}_{\R_I(A)}(\R_I(N)) \rt  H^{i}_{\R_I(A)}(\R_I(N))_{t^{-1}}
  \]
  We note that for $i< d$ by \ref{L} we have an isomorphism  \\ $H^{i}_{\R_I(A)}(\R_I(N)) \cong H^{i-1}_{\R_I(A)}(L^I(N))$. But $L^I(N)_{t^{-1}} = 0$. Thus the assertions (1), (2) follow.

  (3) By an argument similar to  \cite[3.8]{Blanc} it follows that \\ $H^i_{\R_I(A)_+}(\R_I(N))_n = 0$ for all $n \gg 0$. Furthermore by an argument similar to  \cite[4.1]{Blanc}  it follows that $H^i_{\R_I(A)_+}(\R_I(N))_n$ has finite length for all $n \in \Z$. The result follows from (1), (2).

  (4) Let $(-)^\vee$ denote the Matlis dual functor. Take $E = G_I(M) \oplus G_I(A) \oplus (\oplus_{i \leq d}(H^i_\M(\R_I(M))^\vee$. Apply  Lemma \ref{sup} to conclude.

  (5) Note $r \leq d$. We first consider the case when $r \leq d -1$. Then by \cite[5.2]{Pu6},  $H^i_\M(L^I(N))$ has finite length for $i < r$. The result follows from \ref{L}, \ref{compare} and (1).
  Finally we consider the case when $r = d$. Set $E$ to be the Matlis dual of $H^d_\M(\R_I(N))$. Taking cohomology with respect to the exact sequence $0 \rt \R_I(N)(+1) \xrightarrow{t^{-1}} \R_I(N) \rt G_I(N) \rt 0$ and as $H^{d-1}_M(G_I(N))$ has finite length it follows that we have an exact sequence
  \[
  E(+1) \xrightarrow{t^{-1}} E \rt C \rt 0 \quad \text{with $\ell(C) < \infty$}.
  \]
  Let $P$ be homogeneous prime in the support of $E$. As $E$ is $t^{-1}$-torsion it follows that $t^{-1} \in P$. If $P \neq \M$ then localizing the above exact sequence at $P$ and applying Nakayama Lemma we get $E_P = 0$, a contradiction. So $P = \M$ and thus $E$ has finite length.
\end{proof}

\section{Proof of Theorem \ref{main}}
In this section we give:
\begin{proof}[Proof of Theorem \ref{main}]
We may assume that $A$ is complete with infinite residue field (see \ref{red}).
By \ref{hom-ann} there exists $a \geq 1$ such that   $\m^a\sHom_A(M, M) = 0$ where $M = \Syz^A_1(L)$ with $L$ an MCM $A$-module.

Let $I$ be an $\m$-primary ideal of $A$ with $I \subseteq \m^a$. Note we are assuming $H^i(G_I(A))$ have finite length for $i < r$. Note we have nothing to prove if $r \leq 1$. So assume that $t \geq 2$.
By \ref{prelim} we have that $H^i_\M(\R_I(A))$ has finite length for $i < r + 1$.

Let $M = \Syz^A_1(L)$ with $L$ an MCM $A$-module. By \ref{prelim} we may choose $x \in I$-superficial \wrt \ $M$ such that the map
$$H^i_\M(\R_I(M))(-1) \xrightarrow{xt} H^i_\M(\R_I(M))$$
has finite length cokernel for all $i \leq d$.

As $I \subseteq \m^a$ we get that $x\sHom_A(M,M) = 0$. So the multiplication map $\mu_x \colon M \rt M$ factors through a free $A$-module $F$.
\[
\xymatrix{
\
&F
\ar@{->}[dr]^{\beta}
 \\
M
\ar@{->}[ur]_{\alpha}
\ar@{->}[rr]_{\mu_x}
&\
&M
}
\]
It follows that we have a factorization of the multiplication map by $xt^0$ on $\R_I(M)$ as
\[
\xymatrix{
\
&\R_I(F)
\ar@{->}[dr]^{\beta^*}
 \\
\R_I(M)
\ar@{->}[ur]_{\alpha^*}
\ar@{->}[rr]_{\mu_{xt^0}}
&\
&\R_I(M)
}
\]
So for all $i< r+1$ and for all $n \ll 0$ the multiplication map by $x$ on $H^i_\M(\R_I(M))_n$ is zero.
We have a commutative diagram
\[
\xymatrix{
\
&\R_I(M)(+1)
\ar@{->}[dr]^{t^{-1}}
 \\
\R_I(M)
\ar@{->}[ur]_{xt^{1}}
\ar@{->}[rr]_{\mu_{x}}
&\
&\R_I(M)
}
\]

This induces  a commutative diagram in cohomology
\[
\xymatrix{
\
&H^i_\M(\R_I(M))_{n+1}
\ar@{->}[dr]^{t^{-1}}
 \\
H^i_\M(\R_I(M))_{n}
\ar@{->}[ur]_{xt^{1}}
\ar@{->}[rr]_{\mu_{x}}
&\
&H^i_\M(\R_I(M))_{n}
}
\]
By construction of $x$ the top-left map above is surjective for all $n \ll 0$. It follows the top-right map (multiplication by $t^{-1}$) is zero on all components $H^i_\M(\R_I(M))_{n +1}$ for $n \ll 0$ (with $i < r +1$).

We prove $BH^i_\M(G_I(M))_n = 0$ for $n \ll 0$ and all $i < r$.
Note for $i = 0$ we have $BH^0_M(G_I(M))_n \subseteq H^0_\M(G_I(M))_n = 0$ for $n < 0$.
Now assume $i > 0$. We set $H^i(-) = H^i_\M( -)$ on $\R_I(A)$-modules. Set $G = G_I(M), \R =  \R_I(M)$. Note we have exact sequence
\[
H^{i}(\R)_n \xrightarrow{\gamma_n^{i}} H^i(G)_n \xrightarrow{\delta^i_n} H^{i+1}(\R)_{n+1} \xrightarrow{t^{-1}} H^{i + 1}(\R)_n  \xrightarrow{\gamma_n^{i+1}} H^{i+1}(G)_n
\]
Recall $\beta^i_n = \gamma^{i+1}_{n+1} \circ \delta^i_n$.  We also have the map induced by $t^{-1}$ is zero for $n \ll 0$.

We first consider the case when $i = 1$. Then note that $H^1(\R)_n = 0$ for $n < 0$ (see \ref{L}).  So $\delta^1_n$ is injective for $n < 0$. We also have $\gamma_n^2$ is injective for $n \ll 0$. So $\beta^1_n $ is injective for all $n \ll 0$. So $BH^1(G)_n = 0$ for $n \ll 0$.

Now assume $2 \leq i < r$.
Assume  the map  $H^{j+1}(\R)_{n+1} \xrightarrow{t^{-1}} H^{j + 1}(\R)_n$ is zero for all $n \leq n_0$ and for $j = i, i+1 $. So $\gamma^{i+1}_n$ is injective  for all $n \leq n_0$.
We also have $\delta^j_n $ is surjective for all $n \leq n_0$ and $j = i-1, i$.
Fix $n \leq n_0 - 1$. Let $u \in H^i(G)_n$ be such that $\beta^i_n(u) = 0$. As $\gamma^{i+1}_{n+1}$ is injective we have $u \in \ker \delta^i_n = \image \gamma^i_n$.
Say $u = \gamma^i_n(v)$ where $v \in H^i(\R)_n$. However $\delta^{i-1}_{n-1}$ is surjective. So $v = \delta^{i-1}_{n-1}(w)$ for some $w \in H^{i-1}(G)_{n-1}$. It follows that
$u = \beta^{i-1}_{n-1}(w)$. Thus $u$ is a boundary in the Bockstein complex. It follows that $BH^i(G)_n = 0$ for all $n \leq n_0 - 1$.
\end{proof}

\end{document}